\documentclass[a4paper,11pt]{article}
\usepackage[utf8]{inputenc}
\usepackage{a4wide}
\usepackage{algorithm}
\usepackage{algorithmic}
\usepackage{latexsym,amsfonts,amsmath,amssymb,mathrsfs,url,amsthm}
\usepackage{mathtools}
\usepackage{dsfont}
\usepackage{color,graphicx}
\usepackage{lipsum}
\usepackage{hyperref}

\usepackage{xcolor}
\newtheorem{theorem}{Theorem}[section]
\newtheorem{lemma}[theorem]{Lemma}

\newtheorem{remark}[theorem]{Remark}
\newtheorem{proposition}[theorem]{Proposition}
\newtheorem{definition}[theorem]{Definition}
\newtheorem{corollary}[theorem]{Corollary}
\mathtoolsset{showonlyrefs}

\newcommand{\rd}{\, \mathrm{d}}

\newcommand{\bszero}{\boldsymbol{0}}
\newcommand{\bsone}{\boldsymbol{1}}
\newcommand{\bsh}{\boldsymbol{h}}
\newcommand{\bsg}{\boldsymbol{g}}
\newcommand{\bsm}{\boldsymbol{m}}
\newcommand{\bsk}{\boldsymbol{k}}
\newcommand{\bsl}{\boldsymbol{\ell}}
\newcommand{\bsx}{\boldsymbol{x}}
\newcommand{\bsy}{\boldsymbol{y}}

\newcommand{\bsgamma}{\boldsymbol{\gamma}}

\newcommand{\bsDelta}{\boldsymbol{\Delta}}

\newcommand{\NN}{\mathbb{N}}

\newcommand{\RR}{\mathbb{R}}
\newcommand{\ZZ}{\mathbb{Z}}
\newcommand{\Acal}{\mathcal{A}}
\newcommand{\Fcal}{\mathcal{F}}
\newcommand{\Ical}{\mathcal{I}}
\newcommand{\Ocal}{\mathcal{O}}
\newcommand{\Pcal}{\mathcal{P}}

\allowdisplaybreaks

\title{A note on approximation in weighted Korobov spaces via multiple rank-1 lattices\thanks{The work of T.G.\ is supported by JSPS KAKENHI Grant Number 23K03210.}}
\author{Mou Cai\thanks{Graduate School of Engineering, The University of Tokyo, 7-3-1 Hongo, Bunkyo-ku, Tokyo 113-8656, Japan (\url{caimoumou@g.ecc.u-tokyo.ac.jp}; \url{goda@frcer.t.u-tokyo.ac.jp})} \and Takashi Goda\footnotemark[2]}
\date{\today}

\begin{document}

\maketitle

\centerline{\emph{Dedicated to Henryk Wo\'{z}niakowski in celebration of his 80th birthday}}

\sloppy

\begin{abstract}
    This paper studies the multivariate approximation of functions in weighted Korobov spaces using multiple rank-1 lattice rules. It has been shown by K\"{a}mmerer and Volkmer (2019) that algorithms based on multiple rank-1 lattices achieve the optimal convergence rate for the $L_{\infty}$ error in Wiener-type spaces, up to logarithmic factors. While this result was translated to weighted Korobov spaces in the recent monograph by Dick, Kritzer, and Pillichshammer (2022), the analysis requires the smoothness parameter $\alpha$ to be greater than $1$ and is restricted to product weights. In this paper, we extend this result for multiple rank-1 lattice-based algorithms to the case where $1/2<\alpha\le 1$ and for general weights, covering a broader range of periodic functions with low smoothness and general relative importance of variables. We also provide a summability condition on the weights to ensure strong polynomial tractability for any $\alpha>1/2$. Furthermore, by incorporating random shifts into multiple rank-1 lattice-based algorithms, we prove that the resulting randomized algorithm achieves a nearly optimal convergence rate in terms of the worst-case root mean squared $L_2$ error, while retaining the same tractability property.
\end{abstract}


\section{Introduction}

The approximation of multivariate functions is a fundamental task in scientific computing. In this paper, we study the approximation of periodic functions on the $d$-dimensional unit cube $[0,1]^d$ in the weighted Korobov spaces. To tackle high-dimensional problems where the dimension $d$ can be very large, it is often essential to exploit the weighted structure of these spaces \cite{SW98}. By assigning weights $\bsgamma=(\gamma_u)_{u\subset \NN, |u|<\infty}$ to each subset of variables, we can model the relative importance of a group of variables, which is a key ingredient to achieving tractability \cite{NW08,NW10,NW12}; that is, ensuring that the number of function evaluations required to achieve the error tolerance $\varepsilon$ does not grow exponentially fast in $d$. While the precise definition of the weighted Korobov space is deferred to Section~\ref{sec:background}, we note here that the space is characterized by a smoothness parameter $\alpha > 1/2$, which governs the decay rate of the Fourier coefficients. For the sampling points of our approximation algorithm, we employ a union of multiple rank-1 lattice point sets.

To contextualize our contribution, we provide a brief overview of the history of lattice-based approximation, which has seen significant developments over the last two decades. Early works focused on algorithms using a single rank-1 lattice, where the $L_2$ and $L_{\infty}$ errors were studied for spline-based approaches \cite{ZKH09,ZLH06} or for truncated Fourier series approaches \cite{KSW06,KWW09}.
Efficient construction of the underlying generating vectors was investigated in \cite{CKNS20,CKNS21,KSW06}. As described in the recent monograph on lattice rules \cite[Chapters~13 and 14]{DKP22}, the best-known convergence rates of the worst-case $L_2$ and $L_{\infty}$ errors for single lattice approaches are of orders $N^{-\alpha/2+\delta}$ and $N^{-\alpha(2\alpha-1)/(4\alpha-1)+\delta}$, respectively, for arbitrarily small $\delta>0$, where $N$ denotes the number of function evaluations.

However, a fundamental limitation of the single lattice approach was revealed by \cite{BKUV17}, which proved that any algorithm based on a single rank-1 lattice fails to achieve the optimal $L_2$ error rate in periodic Sobolev spaces, including the Korobov space as a special case. While the optimal rate for these spaces should be of order $N^{-\alpha}$ up to logarithmic factors, any deterministic single lattice approach can attain a rate which is at most of order $N^{-\alpha/2}$. We point out that this limitation is, in fact, inherently related to classical results in the theory of number-theoretic methods; similar lower bound arguments can be traced back to the works of Smolyak \cite{Smo60} and Korobov \cite{K63}, as well as the monograph by Hua and Wang \cite[Theorems~9.14 and 9.15]{HW81}. For the reader's convenience, we summarize these classical lower bound arguments in Appendix~\ref{appendix:b}.

One way to rescue this situation is through randomization. Recently, \cite{CGK24} studied the root mean squared $L_2$ error for the truncated Fourier series approach with a randomized single rank-1 lattice, and found an improvement in the convergence rate compared to deterministic results. However, a more fundamental and structural improvement is possible by moving beyond a single lattice to a union of multiple rank-1 lattices \cite{K19, KV19}. The error analysis in \cite{KV19} showed that such algorithms can indeed achieve the optimal convergence rate for the $L_{\infty}$ error in unweighted Wiener-type spaces, up to logarithmic factors. These results were subsequently translated to weighted Korobov spaces in the recent monograph by Dick, Kritzer, and Pillichshammer \cite[Chapter~15]{DKP22}. 

Despite these advances, the following major gaps remain. First, the analysis in \cite{DKP22} is restricted to the high smoothness case $\alpha > 1$ and to product weights. The extension to general weights is particularly motivated by applications to partial differential equations with random coefficients; in such contexts, the optimal choices of weights often do not take a simple product form \cite{GKNSSS15,KKKNS22,KKS20,KSS12}. Second, the conditions imposed on the product weights $\gamma_j$ to ensure strong polynomial tractability in \cite{DKP22} are demanding for small $\alpha < 2$, which can be potentially improved. Moreover, the optimality of the $L_2$ error (or its randomized version) for multiple lattice-based algorithms has not been fully explored in this weighted setting. 

In this paper, we bridge these gaps as follows:
\begin{enumerate}
    \item We prove that the multiple lattice-based algorithm can attain a nearly optimal convergence rate for the worst-case $L_{\infty}$ error for any smoothness $\alpha>1/2$; that is, we extend the result in \cite{DKP22} to the range $1/2<\alpha\le 1$ and to the case of general weights.
    \item By incorporating random shifts into multiple lattice-based algorithms, we prove that the resulting randomized algorithm achieves a nearly optimal convergence rate for the worst-case root mean squared $L_2$ error for any $\alpha>1/2$.
    \item We establish that both the deterministic and randomized error bounds are independent of the dimension $d$ under suitable summability conditions on the general weights. In the case of product weights, the summability condition we obtain is weaker than that required in \cite[Theorem~15.11]{DKP22} when $\alpha<2$.
\end{enumerate}

Before concluding this introduction, we point out that several other approaches have recently been proposed to improve the convergence rates inherent to single lattice-based algorithms \cite{BGKS25,CDG25,PGK25,PKG25}. As the primary objective of this note is to refine and complement the results shown in \cite[Chapter~15]{DKP22}, we focus on the multiple lattice-based algorithms and do not go into the details about these alternative methods.

The rest of this paper is organized as follows. Section~\ref{sec:background} is devoted to introducing the necessary background, including the weighted Korobov space and rank-1 lattice point sets. In Section~\ref{sec:algorithm}, we formally introduce the multiple lattice-based algorithm from \cite{K19,KV19}. The error analysis is conducted in Section~\ref{sec:error_analysis}, where we present the proof of our main results stated above.

\section{Background and notation}\label{sec:background}

\paragraph{Notation.} Throughout this paper, let $\ZZ$ denote the set of integers and $\NN$ the set of positive integers. Boldface lowercase letters, such as $\bsx$, represent vectors (e.g., $\bsx = (x_1, \dots, x_d) \in [0,1]^d$) whose dimensions are clear from the context unless otherwise stated. For $d\in \NN$, we write $[1{:}d] \coloneqq \{1, \dots, d\}$. For a subset $u \subseteq [1{:}d]$, we denote its complement by $-u \coloneqq [1{:}d] \setminus u$. For a vector $\bsx \in \RR^d$, $\bsx_u \coloneqq (x_j)_{j \in u}$ represents the subvector containing the components of $\bsx$ indexed by $u$. We also use the notation $(\bsx_u, \bszero)$ to denote the vector in $\RR^d$ whose $j$-th component is $x_j$ if $j \in u$ and $0$ if $j \in -u$.

\subsection{Weighted Korobov space}\label{subsec:Korobov}

Any function $f\in L_2([0,1]^d)$ can be formally expanded into a Fourier series:
\[ f(\bsx)\sim \sum_{\bsk\in \ZZ^d}\widehat{f}(\bsk)\, e^{2\pi i \bsk\cdot \bsx}, \]
where the Fourier coefficients $\widehat{f}(\bsk)$ are given by
\[ \widehat{f}(\bsk) = \int_{[0,1]^d}f(\bsx)\, e^{-2\pi i \bsk\cdot \bsx}\rd \bsx, \]
and the notation $\sim$ denotes equality in the $L_2$ sense.

Let $\alpha>1/2$ be a smoothness parameter. To model the relative importance of different subsets of variables, we introduce a collection of general weights $\bsgamma=(\gamma_u)_{u\subset \NN, |u|<\infty}$, where $\gamma_u\ge 0$ for all finite $u\subset \NN$, with the convention $\gamma_{\emptyset}=1$. The weighted Korobov space $H_{d,\alpha,\bsgamma}$ is the reproducing kernel Hilbert space (RKHS) with the reproducing kernel $K_{d,\alpha,\bsgamma}$:
\begin{align*}
    K_{d,\alpha,\bsgamma}(\bsx, \bsy) \coloneqq \sum_{u \subseteq [1{:}d]} \gamma_u^2 \prod_{j \in u} \left( \sum_{k_j \in \ZZ \setminus \{0\}} \frac{e^{2\pi i k_j(x_j - y_j)}}{|k_j|^{2\alpha}} \right).
\end{align*}
Equivalently, this space consists of periodic functions $f$ such that the following norm is finite:
\begin{align*}
    \|f\|_{d,\alpha,\bsgamma} \coloneqq \left( \sum_{\bsk \in \ZZ^d} r_{\alpha, \bsgamma}^2(\bsk) |\widehat{f}(\bsk)|^2 \right)^{1/2},
\end{align*}
where for $\bsk \in \ZZ^d$, the weight $r_{\alpha, \bsgamma}(\bsk)$ is defined using $u = \text{supp}(\bsk) \coloneqq \{j \in [1{:}d] \mid k_j \neq 0\}$ as
\[ r_{\alpha, \bsgamma}(\bsk) \coloneqq \frac{1}{\gamma_u} \prod_{j \in u} |k_j|^\alpha. \] 
We use the convention that an empty product equals $1$, and if $\gamma_u = 0$ for some $u$, we require $\widehat{f}(\bsk) = 0$ for all $\bsk$ with $\text{supp}(\bsk) = u$. Under the condition $\alpha > 1/2$, the series in the kernel converges uniformly, and the functions in $H_{d,\alpha,\bsgamma}$ are continuous and periodic in each variable.

Regarding the weights $\bsgamma$, several special forms have been considered in the literature. For sequences of non-negative numbers $(\gamma_j)_{j \ge 1}$ and $(\Gamma_\ell)_{\ell \ge 0}$, common examples include
\begin{itemize}
    \item Product weights: 
    \[ \gamma_u = \prod_{j\in u}\gamma_j.\]
    \item \emph{Product and order-dependent} (POD) weights:
    \[ \gamma_u = \Gamma_{|u|}\prod_{j\in u}\gamma_j. \]
    \item \emph{Smoothness-driven product and order-dependent} (SPOD) weights: For a double sequence $(\gamma_{j,m})_{j \ge 1, m \in [1{:}\sigma]}$ with $\sigma \in \NN$,
    \[ \gamma_u = \sum_{\bsm_u\in [1{:}\sigma]^{|u|}}\Gamma_{|\bsm_u|_1}\prod_{j\in u}\gamma_{j,m_j}, \]
    where $|\bsm_u|_1=\sum_{j\in u}m_j$.
\end{itemize}
The POD and SPOD weights naturally arise in the analysis of partial differential equations with random coefficients \cite{GKNSSS15, KKKNS22, KKS20, KSS12}.

\subsection{Rank-1 lattice rules}\label{subsec:lattice}

Next, we introduce the sampling points used in our approximation algorithm.
\begin{definition}[Rank-1 lattice point set]
    For a given number of points $N\in \NN$ and a generating vector $\bsg=(g_1,\ldots,g_d)\in \{1,\ldots,N-1\}^d$, the rank-1 lattice point set $P_{N,\bsg}$ is defined by
    \[ P_{N,\bsg}\coloneqq\left\{ \bsy_n=\left(\left\{ \frac{ng_1}{N}\right\},\ldots,\left\{ \frac{ng_d}{N}\right\}\right) \,\mid \, 0\leq n<N\right\}.\]
\end{definition}

The properties of lattice rules are closely related to the following set of vectors in the frequency domain.

\begin{definition}[Dual lattice]
    The dual lattice $P_{N,\bsg}^\perp$ for a rank-1 lattice point set with generating vector $\bsg$ and size $N$ is defined as
    \[ P_{N,\bsg}^\perp\coloneqq\left\{\bsk \in \ZZ^{d} \mid \bsk \cdot \bsg \equiv 0 \pmod N\right\}. \]
\end{definition}

The following lemma describes the well-known character property of lattice point sets, which is one of the fundamental equalities for the error analysis of QMC rules.

\begin{lemma}[Character property]\label{lem:character}
    Let $P_{N,\bsg}=\{\bsy_0,\ldots,\bsy_{N-1}\}$ be a rank-1 lattice point set. For any $\bsk\in \ZZ^d$, it holds that
    \[ \frac{1}{N}\sum_{n=0}^{N-1} e^{2\pi i \bsk\cdot \bsy_n}= \mathds{1}_{\bsk \in P_{N, \bsg}^{\perp}}, \]
    where $\mathds{1}_{\cdot}$ denotes the indicator function.
\end{lemma}

\subsection{Weighted hyperbolic cross}\label{subsec:hyperbolic_cross}

For a function $f \in H_{d,\alpha,\bsgamma}$, we consider the truncated Fourier expansion:
\begin{align} \label{eq:truncated_Fourier}
f_{\Ical}(\bsx) = \sum_{\bsk\in \Ical}\widehat{f}(\bsk)\, e^{2\pi i\bsk\cdot \bsx},
\end{align}
where $\Ical\subset\ZZ^d$ is a finite index set. Then, the truncated Fourier series approach \cite{KSW06,KWW09,K19,KV19} for approximating $f$ employs a suitably chosen quadrature rule to estimate the Fourier coefficients $\widehat{f}(\bsk)$ for all $\bsk\in \Ical$.

Here, it is clear that the index set $\Ical$ must be carefully chosen. On the one hand, we aim to keep its cardinality $|\Ical|$ small to minimize computational costs. On the other hand, to ensure a fast convergence rate of the error, $\Ical$ should include indices $\bsk$ that correspond to the largest Fourier coefficients. In view of these requirements, we define $\Ical$ by a weighted hyperbolic cross of the form:
\begin{align}\label{eq:weight_hyper}
\Acal_d\coloneqq \Acal_{d,\alpha,\bsgamma}(M)=\left\{\bsk \in \mathbb{Z}^d \mid r_{\alpha,\bsgamma} (\bsk)\leq M  \right\},
\end{align}
for some $M>0$, where $\alpha>1/2$, $d \in \NN$, and $\bsgamma=(\gamma_u)_{u\subset \NN, |u|<\infty}$ is a collection of non-negative weights. Throughout this paper, we assume $\Ical = \Acal_d$ unless otherwise stated. We refer to the positive constant $M$ as the \emph{radius} and $|\Acal_d|$ as the \emph{size} of the hyperbolic cross.

The relationship between the radius $M$ and the size $|\Acal_d|$ is shown in the following lemma. While similar results restricted to product weights are available, for instance, in \cite[Lemma~13.1]{DKP22}, the following statement for general weights is given in \cite{CKNS20}. We provide a brief proof here.

\begin{lemma}\label{lem:size_hyperbolic_cross}
    For $\alpha > 1/2$, $d \in \NN$, a collection of non-negative weights $\bsgamma=(\gamma_u)_{u\subset \NN, |u|<\infty}$, and $M > 0$, let $\Acal_d$ be defined as in \eqref{eq:weight_hyper}. Then, for any $\lambda > 1/\alpha$, we have
    \begin{align*}
        |\Acal_d| \leq M^{\lambda} \sum_{u\subseteq [1{:}d]}\gamma_u^{\lambda}(2\zeta(\alpha\lambda))^{|u|},
    \end{align*}
    where $\zeta(s)=\sum_{k=1}^{\infty}k^{-s}$ denotes the Riemann zeta function.
\end{lemma}

\begin{proof}
By the definition of $\Acal_d$ and the fact that $\mathds{1}_{x\leq y}\le (y/x)^{\lambda}$ for any $x,y,\lambda > 0$, we have
\begin{align*}
    |\Acal_d| & = \sum_{\bsk\in \ZZ^d}\mathds{1}_{r_{\alpha,\bsgamma} (\bsk)\leq M} \le \sum_{\bsk\in \ZZ^d}\left( \frac{M}{r_{\alpha,\bsgamma} (\bsk)}\right)^{\lambda}\\
    & = M^{\lambda}\sum_{u\subseteq [1{:}d]}\sum_{\bsk_u\in (\ZZ\setminus \{0\})^{|u|}}\frac{1}{r^{\lambda}_{\alpha,\bsgamma} (\bsk_u,\bszero)}\\
    & = M^{\lambda}\sum_{u\subseteq [1{:}d]}\gamma_u^{\lambda}\left( \sum_{k\in \ZZ\setminus \{0\}}\frac{1}{|k|^{\alpha \lambda}}\right)^{|u|} = M^{\lambda} \sum_{u\subseteq [1{:}d]}\gamma_u^{\lambda}(2\zeta(\alpha\lambda))^{|u|},
\end{align*}
where the condition $\alpha \lambda > 1$ ensures the absolute convergence of the Riemann zeta function.
\end{proof}

Additionally, we prove the following auxiliary result related to the weighted hyperbolic cross $\Acal_d$. While similar estimates exist in the literature (e.g., \cite[Lemma~14.2]{DKP22}), we provide a proof for the sake of completeness, extending the result to the case of general weights.
\begin{lemma}\label{lem:sum_weights_hyperbolic_cross}
    For $\alpha > 1/2$, $d \in \NN$, a collection of non-negative weights $\bsgamma=(\gamma_u)_{u\subset \NN, |u|<\infty}$, and $M \ge 1$, let $\Acal_d$ be defined as in \eqref{eq:weight_hyper}. Then, for any $\lambda \in (1/\alpha,2)$, we have
    \[ \sum_{\bsk\notin \Acal_d}\frac{1}{r_{\alpha,\bsgamma}^2(\bsk)} \le \frac{1}{M^{2-\lambda}}\,\frac{8(3-\lambda)}{2-\lambda}\sum_{u\subseteq [1{:}d]}\gamma_u^{\lambda}(2\zeta(\alpha\lambda))^{|u|}. \]
\end{lemma}

\begin{proof}
    By definition, any element $\bsk\notin \Acal_d$ satisfies $r_{\alpha,\bsgamma}(\bsk)>M$. By partitioning the elements $\bsk\notin \Acal_d$ according to the magnitude of $r_{\alpha,\bsgamma}(\bsk)$, we have
\begin{align*}
    \sum_{\bsk\notin \Acal_d}\frac{1}{r_{\alpha,\bsgamma}^2(\bsk)} & \le  \sum_{\ell=\lfloor M \rfloor}^{\infty}\sum_{\substack{\bsk\in \ZZ^d\\ \ell\le r_{\alpha,\bsgamma}(\bsk)<\ell+1}}\frac{1}{r_{\alpha,\bsgamma}^2(\bsk)} \\
    & \le \sum_{\ell=\lfloor M\rfloor}^{\infty}\frac{1}{\ell^2}\left(\sum_{\bsk\in \ZZ^d}\mathds{1}_{r_{\alpha,\bsgamma}(\bsk)<\ell+1}-\sum_{\bsk\in \ZZ^d}\mathds{1}_{r_{\alpha,\bsgamma}(\bsk)<\ell} \right)\\
    & \le \sum_{\ell=\lfloor M\rfloor}^{\infty}\left(\frac{1}{\ell^2}-\frac{1}{(\ell+1)^2}\right) \sum_{\bsk\in \ZZ^d}\mathds{1}_{r_{\alpha,\bsgamma}(\bsk)<\ell+1}\\
    & \le 8\sum_{\ell=\lfloor M\rfloor}^{\infty}\frac{1}{(\ell+1)^3} \sum_{\bsk\in \ZZ^d}\mathds{1}_{r_{\alpha,\bsgamma}(\bsk)<\ell+1}.
\end{align*}

Regarding the sum over $\bsk$ above, it follows from Lemma~\ref{lem:size_hyperbolic_cross} that
\[ \sum_{\bsk\in \ZZ^d}\mathds{1}_{r_{\alpha,\bsgamma}(\bsk)<\ell+1} \le |\Acal_{d,\alpha,\bsgamma}(\ell+1)| \le (\ell+1)^{\lambda}\sum_{u\subseteq [1{:}d]}\gamma_u^{\lambda}(2\zeta(\alpha\lambda))^{|u|},\]
for any $\ell\in \NN$ and $\lambda>1/\alpha$. Substituting this into the previous inequality, and assuming $\lambda<2$, we get
\begin{align*}
    \sum_{\bsk\notin \Acal_d}\frac{1}{r_{\alpha,\bsgamma}^2(\bsk)} & \le 8\sum_{u\subseteq [1{:}d]}\gamma_u^{\lambda}(2\zeta(\alpha\lambda))^{|u|}\sum_{\ell=\lfloor M\rfloor}^{\infty}\frac{1}{(\ell+1)^{3-\lambda}} \\
    & \le 8\sum_{u\subseteq [1{:}d]}\gamma_u^{\lambda}(2\zeta(\alpha\lambda))^{|u|}\left( \frac{1}{(\lfloor M\rfloor+1)^{3-\lambda}}+\int_{\lfloor M\rfloor}^{\infty}\frac{1}{(x+1)^{3-\lambda}}\rd x\right)\\
    & = 8\sum_{u\subseteq [1{:}d]}\gamma_u^{\lambda}(2\zeta(\alpha\lambda))^{|u|}\left( \frac{1}{(\lfloor M\rfloor+1)^{3-\lambda}}+\frac{1}{2-\lambda}\frac{1}{(\lfloor M\rfloor+1)^{2-\lambda}}\right)\\
    & \le \frac{1}{M^{2-\lambda}}\, \frac{8(3-\lambda)}{2-\lambda}\sum_{u\subseteq [1{:}d]}\gamma_u^{\lambda}(2\zeta(\alpha\lambda))^{|u|},
\end{align*}
which completes the proof.
\end{proof}

\section{Multiple rank-1 lattice-based algorithms}\label{sec:algorithm}

In this section, we start with the single lattice-based algorithm from \cite{KSW06,KWW09}, and then formally introduce the multiple lattice-based algorithm from \cite{K19,KV19}. 

\subsection{Single lattice-based algorithms}

The single lattice-based approach approximates $f \in H_{d,\alpha,\bsgamma}$ by combining the truncated Fourier series on the index set $\Acal_d$ with the discretization of the coefficients $\widehat{f}(\bsk)$ using a rank-1 lattice $P_{N,\bsg}=\{\bsy_n\mid 0\le n<N\}$. For any $\bsk\in \Acal_d$, the discrete Fourier coefficient $\widehat{f}_{N}(\bsk)$ is computed as
\begin{align}\label{eq:aliasing_formula}
\widehat{f}_{N}(\bsk) = \frac{1}{N}\sum_{n=0}^{N-1}f(\bsy_n) \, e^{-2\pi i\bsk\cdot \bsy_n}= \widehat{f}(\bsk) + \sum_{\bsl \in P_{N, \bsg}^\perp \setminus \{ \bszero \}} \widehat{f}(\bsl + \bsk),
\end{align}
where the second equality follows from the character property of the lattice (cf.~Lemma~\ref{lem:character}).
The resulting approximation operator $A^{\text{sing}}_N$ is defined by
\begin{align*}
A_N^{\text{sing}}(f)(\bsx) \coloneqq \sum_{\bsk\in \Acal_d}\widehat{f}_N(\bsk)\, e^{2\pi i\bsk\cdot \bsx}.
\end{align*}

The fundamental limitation of this approach lies in the aliasing issue. For any $\bsk\in \ZZ^d$, it follows from \eqref{eq:aliasing_formula} that
\[ \widehat{f}_N(\bsk-\bsl)=\widehat{f}_N(\bsk),\]
whenever $\bsl\in P^{\perp}_{N,\bsg}$. This implies that the frequency component $\bsk$ is indistinguishable from any other component $\bsk-\bsl$ shifted by a dual lattice vector, based solely on the function values at the lattice points. To uniquely identify the coefficients on the index set $\Acal_d$, the dual lattice $P_{N, \bsg}^\perp$ must avoid the non-zero elements of the ``difference set'' $\Acal_d - \Acal_d = \{\bsk-\bsl \mid \bsk,\bsl\in \Acal_d\}$. While this can be achieved if $N$ is sufficiently large relative to $M$, a smaller $M$ inevitably leads to a larger truncation error, as the coefficients $\widehat{f}(\bsk)$ for $\bsk\notin \Acal_d$ are ignored in the approximation $A_N^{\text{sing}}$. This inherent trade-off results in a non-optimal convergence rate for the single lattice-based algorithm in both the $L_2$ and $L_{\infty}$ norms. A more detailed discussion on the classical arguments for these lower bounds that explain the non-optimality of the single lattice-based approach is provided in Appendix~\ref{appendix:b}.

\subsection{Multiple lattice-based algorithms}

To overcome the aliasing issues inherent in the single lattice-based approach for large hyperbolic crosses, a reconstruction strategy using a union of rank-1 lattices was proposed in \cite{K18, KV19}.

The multiple lattice-based algorithm approximates the Fourier coefficients in the truncated expansion \eqref{eq:truncated_Fourier} (with $\Ical=\Acal_d$) by aggregating information from $L$ separate rank-1 lattice point sets $P_{n_1,\bsg_1}, \dots, P_{n_L,\bsg_L}$. The total number of sampling points is $N = \sum_{t=1}^L n_t$. In practice, the sizes $n_t$ are typically chosen to be roughly equal to balance the computational load across the lattices.

The core of this approach is a coverage strategy: we partition $\Acal_d$ based on aliasing properties. Specifically, we first identify a subset $\Acal_{d,1} \subset \Acal_d$ whose coefficients can be approximated by the first lattice $P_{n_1, \bsg_1}$ without aliasing; that is, for any $\bsk\in \Acal_{d,1}$ and $\bsl\in P^{\perp}_{n_1,\bsg_1}\setminus \{\bszero\}$, we have $\bsk+\bsl\notin \Acal_d$. For the remaining indices $\Acal_d \setminus \Acal_{d,1}$, we employ a second lattice $P_{n_2, \bsg_2}$, and repeat this process until every frequency in $\Acal_d$ is covered by at least one lattice that is ``aliasing-free'' for that specific index. Ideally, the number of lattices $L$ should satisfy $L \approx \log |\Acal_d|$, a property shown to be achievable in \cite{KV19}. Following their results, we provide a specific probabilistic construction for these lattices in Appendix~\ref{appendix:a}. Furthermore, in Corollary~\ref{cor:number_sampling}, we provide an upper bound on the total number of points $N$ required to cover $\Acal_d$ while maintaining this aliasing-free property.

To formally analyze the aliasing behavior within each lattice $P_{n_t, \bsg_t}$, we introduce the concept of a \emph{fiber}, which collects all frequencies in $\Acal_d$ that are aliased under a given lattice rule.

\begin{definition}[Fiber]
Let $\Acal_d$ be a finite frequency set and $P_{n, \bsg}$ be a rank-1 lattice point set. For any $\bsk\in \ZZ^d$, the fiber of $\bsk$ with respect to $\Acal_d$ and $P_{n,\bsg}$ is defined as
\begin{align}\label{eq:fiber}
\Fcal_{\Acal_d, n}(\bsk, \bsg) \coloneqq \left\{ \bsm \in \Acal_d \mid \bsm \cdot \bsg \equiv \bsk \cdot \bsg \pmod{n} \right\}.
\end{align}
\end{definition}

A frequency $\bsk \in \Acal_d$ is said to be \emph{aliasing-free} with respect to $P_{n, \bsg}$ if $|\Fcal_{\Acal_d, n}(\bsk, \bsg)| = 1$. To aggregate the information from $L$ lattices, for each $t\in \{1,\ldots,L\}$ and $\bsk\in \Acal_d$, we define the indicator function $U_t(\bsk)$ by
\[ U_t(\bsk)=\begin{cases}
1 & \text{if $|\Fcal_{\Acal_d, n_t}(\bsk, \bsg_t)| = 1$,}\\
0 & \text{otherwise.}
\end{cases}\]
Let $\xi(\bsk) \coloneqq \sum_{t=1}^L U_t(\bsk)$ be the number of available aliasing-free lattices for $\bsk$. The construction algorithm (detailed in Appendix~\ref{appendix:a}) ensures that $\xi(\bsk) \ge 1$ for all $\bsk \in \Acal_d$ while keeping $L$ small.

Verifying the indicator function $U_t(\bsk)$ is equivalent to checking the following condition:
\begin{align*}
\bsk\cdot\bsg_t\not\equiv \bsm\cdot\bsg_t \pmod{n_t}\quad \text{for all $\bsm\in \Acal_d\setminus\{\bsk \}$}.
\end{align*}
To evaluate $U_t(\bsk)$ efficiently for all $t\in \{1,\ldots,L\}$ and $\bsk\in \Acal_d$, we can employ a simple counting (or binning) strategy. For a fixed $t$, note that the values $\nu_{\bsk,t}\coloneqq \bsk\cdot \bsg_t\pmod {n_t}$ are integers in $\{0,1,\ldots,n_t-1\}$. We first allocate an integer array of length $n_t$ initialized to all zeros. For each $\bsk\in \Acal_d$, we compute $\nu_{\bsk,t}$ (which requires $\Ocal(d)$ operations per frequency) and increment the corresponding entry in the array. This step takes $\Ocal(d|\Acal_d|)$ operations. Then, a second pass over $\Acal_d$ allows us to determine $U_t(\bsk)$ in $\Ocal(|\Acal_d|)$ time: $U_t(\bsk)=1$ if the array entry at index $\nu_{\bsk,t}$ is exactly $1$, and $0$ otherwise. Thus, the cost for a single lattice is $\Ocal(d|\Acal_d|+n_t)$. Summing this over all $L$ lattices, the total computational complexity becomes
\[ \Ocal\left(dL|\Acal_d|+\sum_{t=1}^{L}n_t\right)=\Ocal(dL|\Acal_d|+N).\]
Since $N$ is bounded by $2cL_{\max}(|\Acal_d|-1)$ according to Proposition~\ref{prop:complexity} and $L_{\max}$ is chosen as in Algorithm~\ref{Algorithm_2}, the overall complexity simplifies to $\Ocal(dL_{\max}|\Acal_d|)=\Ocal(d|\Acal_d|\log |\Acal_d|)$. This efficiency ensures that the precomputation of the aliasing-free index sets remains feasible even in high dimensions.

The multiple lattice-based approximation $A_N^{\mathrm{mult}}(f)$ is then defined as:
\begin{align}\label{eq:mult_algorithm}
A_N^{\mathrm{mult}}(f)(\bsx)=\sum_{\bsk\in\Acal_d} \widehat{f}^{\mathrm{mult}}_{N}(\bsk)\, e^{2\pi i\bsk\cdot \bsx} ,
\end{align}
where the discrete Fourier coefficients $\widehat{f}^{\mathrm{mult}}_{N}(\bsk)$ are computed by averaging the aliasing-free results:
\begin{align*}
\widehat{f}_N^{\mathrm{mult}}(\bsk) \coloneqq \frac{1}{\xi(\bsk)}\sum_{t=1}^L \frac{U_{t}(\bsk)}{n_{t}}\sum_{i=0}^{n_t-1}f(\bsy^{(t)}_i) \, e^{-2\pi i\bsk\cdot \bsy^{(t)}_i}.
\end{align*}
Here, $\{\bsy^{(t)}_0,\ldots,\bsy^{(t)}_{n_t-1}\}$ denotes the $t$-th rank-1 lattice point set for $t\in \{1,\ldots,L\}$.

Furthermore, we consider a randomized version of this algorithm. Let $\bsDelta\in [0,1]^d$ be a random shift chosen uniformly at random. The randomized multiple lattice-based algorithm is similarly defined as
 \begin{align}\label{eq:randomized_algorithm}
A^{\mathrm{mult}}_{N,\bsDelta}(f)(\bsx)=\sum_{\bsk\in\Acal_d} \widehat{f}^{\mathrm{mult}}_{N,\bsDelta}(\bsk) \, e^{2\pi i\bsk\cdot \bsx},
 \end{align}
where the discrete Fourier coefficients $\widehat{f}_{N, \bsDelta}^{\mathrm{mult}}(\bsk)$ are computed using the shifted sampling points:
\begin{align*}
\widehat{f}_{N,\bsDelta}^{\mathrm{mult}}(\bsk)\coloneqq \frac{1}{\xi(\bsk)}\sum_{t=1}^L \frac{U_{t}(\bsk)}{n_{t}}\sum_{i=0}^{n_t-1}f(\bsy^{(t)}_{i,\bsDelta}) \, e^{-2\pi i\bsk\cdot \bsy^{(t)}_{i,\bsDelta}}.
\end{align*}
Here, $\bsy^{(t)}_{i,\bsDelta} \coloneqq \{\bsy^{(t)}_{i}+\bsDelta\}$ denotes the shifted lattice point, where the fractional part is taken componentwise.

\section{Error analysis}\label{sec:error_analysis}

In this section, we provide a theoretical error analysis for the deterministic algorithm $A_N^{\mathrm{mult}}$ and its randomized counterpart $A^{\mathrm{mult}}_{N,\bsDelta}$. Throughout the analysis, the index set $\Acal_d$ is chosen as the weighted hyperbolic cross defined in \eqref{eq:weight_hyper}. To facilitate the analysis across multiple lattices, for each $t\in \{1,\ldots,L\}$, we define the subset of indices that are aliasing-free with respect to the $t$-th lattice as
\begin{align*}
\Acal_{d,t}\coloneqq \left\{\bsk \in \Acal_d \mid U_t(\bsk)=1 \right\} \subseteq \Acal_d.
\end{align*}

We assume that the collection of rank-1 lattices $\{P_{n_t, \bsg_t}\}_{t=1}^L$ is obtained via the construction algorithm described in Appendix~\ref{appendix:a} (Algorithm~\ref{Algorithm_2}), which ensures that $\xi(\bsk) \ge 1$ for all $\bsk \in \Acal_d$ with probability $1-\delta$, where $\delta\in (0,1)$ is one of the input parameters in Algorithm~\ref{Algorithm_2}. Our error bounds rely on the following fundamental property of these aliasing-free sets and their relation to the dual lattice (cf.~\cite[Proof of Lemma~3.1]{KV19} and \cite[Proof of Lemma~15.6]{DKP22}).

\begin{lemma}\label{lem:multiple_rank-1_property}
Let $\Acal_d$ be the hyperbolic cross defined in \eqref{eq:weight_hyper} with radius $M>0$, and let $\{P_{n_t, \bsg_t}\}_{t=1}^L$ be the multiple rank-1 lattice point sets obtained by Algorithm~\ref{Algorithm_2}. For any $t\in \{1,\ldots,L\}$ and any non-zero vector $\bsl\in P_{n_t,\bsg_t}^{\perp}\setminus\{ \bszero \}$, the following properties hold:
\begin{enumerate}
    \item For any $\bsk\in \Acal_{d,t}$, we have
        \begin{align*}
            \bsk+\bsl\notin \Acal_d.
        \end{align*}
    \item The sets
        \begin{align*}
            \{\bsk+\bsl \mid \bsl\in   P_{n_t,\bsg_t}^{\perp}\setminus\{ \bszero \} \}
        \end{align*}
        for distinct $\bsk\in \Acal_{d,t}$ are pairwise disjoint.
\end{enumerate}
\end{lemma}

\begin{proof}
To prove the first assertion, suppose there exists $\bsk\in \Acal_{d,t}$ and $\bsl\in P_{n_t,\bsg_t}^{\perp}\setminus\{ \bszero \}$, such that $\bsk'\coloneqq\bsk+\bsl\in \Acal_{d}$. Since  $\bsl\in P_{n_t,\bsg_t}^{\perp}$, it follows that $\bsl\cdot \bsg_t\equiv 0 \pmod {n_t}$, which implies $\bsk'\cdot \bsg_t\equiv \bsk\cdot \bsg_t \pmod {n_t}$. This results in $\bsk' \in \Fcal_{\Acal_d, n_t}(\bsk, \bsg_t)$. Since $\bsk' \neq \bsk$, we have $|\Fcal_{\Acal_d, n_t}(\bsk, \bsg_t)| \ge 2$, which directly contradicts the assumption that $\bsk \in \Acal_{d,t}$.

To prove the second assertion, assume there exist distinct $\bsk, \bsk' \in \Acal_{d,t}$ and elements $\bsl, \bsl' \in P_{n_t, \bsg_t}^\perp \setminus \{ \bszero \}$ such that $\bsk + \bsl = \bsk' + \bsl'$. Then we have $\bsk - \bsk' = \bsl' - \bsl\in P_{n_t, \bsg_t}^\perp\setminus \{\bszero\}$. This implies $\bsk \cdot \bsg_t \equiv \bsk' \cdot \bsg_t \pmod{n_t}$. This again yields $|\Fcal_{\Acal_d, n_t}(\bsk, \bsg_t)| \ge 2$, a contradiction.
\end{proof}

With this lemma, we are now ready to derive the worst-case error bounds. We divide our analysis into two parts: the $L_\infty$ error for the deterministic algorithm and the root mean squared $L_2$ error for the randomized algorithm.

\subsection{The worst-case \texorpdfstring{$L_{\infty}$}{L-infinity} error}\label{subsec:L_infty}

We begin by defining the worst-case error of the deterministic multiple lattice-based algorithm $A_N^{\mathrm{mult}}$ in the $L_\infty$ norm:
\[ \mathrm{err}_{\infty}(A_N^{\mathrm{mult}},H_{d,\alpha,\bsgamma}) \coloneqq \sup_{\substack{f\in H_{d,\alpha,\bsgamma}\\ \| f\|_{d,\alpha,\bsgamma}\le 1}}\|f-A_N^{\mathrm{mult}}(f)\|_{L_{\infty}}, \]
where the $L_{\infty}$ error for an individual function is simply given by
\[ \|f-A_N^{\mathrm{mult}}(f)\|_{L_{\infty}}=\sup_{\bsx\in [0,1]^d}\left|f(\bsx)-A_N^{\mathrm{mult}}(f)(\bsx)\right|.\]

As one of the main results of this paper, we prove an upper bound on the worst-case $L_{\infty}$ error $\mathrm{err}_{\infty}(A_N^{\mathrm{mult}},H_{d,\alpha,\bsgamma})$ for any fixed smoothness $\alpha>1/2$ and general weights $\bsgamma$, extending the result shown in \cite[Theorem~15.8]{DKP22}.

\begin{theorem}\label{thm:main_result_deterministic}
    For $\alpha > 1/2$, $d \in \NN$, a collection of non-negative weights $\bsgamma=(\gamma_u)_{u\subset \NN, |u|<\infty}$, and $M \ge 1$, let $\Acal_d$ be defined as in \eqref{eq:weight_hyper}. Consider the deterministic algorithm $A_N^{\mathrm{mult}}$ of the form \eqref{eq:mult_algorithm}, where the number of lattices $L$ and the lattice parameters $(n_t,\bsg_t)$ for $t=1,\ldots,L$ are obtained by Algorithm~\ref{Algorithm_2}. Provided that the parameters $c\in (1,\infty)$ and $\delta\in (0,1)$ are chosen such that the assumptions in Proposition~\ref{prop:complexity} are satisfied, then for any $\lambda\in (1/\alpha,2)$, we have
    \[ \mathrm{err}_{\infty}(A_N^{\mathrm{mult}},H_{d,\alpha,\bsgamma}) \le \frac{L+1}{M^{1-\lambda/2}} \left(\frac{8(3-\lambda)}{2-\lambda}\sum_{u\subseteq [1{:}d]}\gamma_u^{\lambda}(2\zeta(\alpha\lambda))^{|u|}\right)^{1/2}, \]
with probability $1-\delta$ (with respect to the construction of the lattices).
\end{theorem}

\begin{remark}\label{rem:comparison_index_set}
It is important to note that our notation and choice of index set differ slightly from those used in \cite[Chapter 15]{DKP22}. Specifically, in \cite{DKP22}, the smoothness parameter (denoted there as $\tilde{\alpha}\coloneqq \alpha+\lambda$) is assumed to be strictly greater than $1$ since they assume $\alpha,\lambda>1/2$. In contrast, we consider a single parameter $\alpha>1/2$, thereby extending the analysis to spaces with lower regularity. Furthermore, our weight $\gamma_u^2$  corresponds to $\gamma_u$ in the notation of \cite{DKP22}.

Our choice \eqref{eq:weight_hyper} of the index set $\Acal_d$ also differs from the one used in \cite[Theorem~15.8]{DKP22}, which was defined as
\[ \Acal_d = \Acal_{d,1,\bsgamma^{1/2}}(M)=\left\{\bsk \in \mathbb{Z}^d \mid r_{1,\bsgamma^{1/2}} (\bsk)\leq M  \right\}, \] 
where we write $\bsgamma^{1/2}=(\gamma_u^{1/2})_{u\subset \NN, |u|<\infty}$. Our extension to lower smoothness $1/2<\alpha\le 1$ stems essentially from our $\alpha$-dependent choice of $\Acal_d$. Conversely, since the index set in \cite{DKP22} does not depend on $\alpha$, their multiple lattice-based algorithm possesses the advantage of being universal with respect to the smoothness parameter.
\end{remark}

\begin{proof}[Proof of Theorem~\ref{thm:main_result_deterministic}]
It suffices to show that for any individual function $f\in H_{d,\alpha,\bsgamma}$, the inequality
\[ \| f-A_N^{\mathrm{mult}}(f)\|_{L_{\infty}} \le \|f\|_{d,\alpha,\bsgamma}\frac{L+1}{M^{1-\lambda/2}}\left(\frac{8(3-\lambda)}{2-\lambda}\sum_{u\subseteq [1{:}d]}\gamma_u^{\lambda}(2\zeta(\alpha\lambda))^{|u|}\right)^{1/2} \]
holds. Using the Fourier series expansion of $f$ and the definition of the approximation algorithm \eqref{eq:mult_algorithm}, the $L_{\infty}$ error is bounded by
\begin{align}
\| f-A_N^{\mathrm{mult}}(f)\|_{L_{\infty}} & = \sup_{\bsx\in [0,1]^d}\left|\sum_{\bsk\notin\Acal_d} \widehat{f}(\bsk) \, e^{2\pi i\bsk\cdot \bsx}+\sum_{\bsk\in\Acal_d} \left(\widehat{f}(\bsk)-\widehat{f}^{\mathrm{mult}}_{N}(\bsk)\right)  e^{2\pi i\bsk\cdot \bsx}\right| \notag \\
& \leq \sum_{\bsk \notin\Acal_d}|\widehat{f}(\bsk)|+ \sum_{\bsk\in\Acal_d}\left|\widehat{f}(\bsk)- \widehat{f}_N^{\mathrm{mult}}(\bsk)  \right| .\label{eq:inf_error_analytic}
\end{align}
The bound
\[ \sum_{\bsk\in\Acal_d}\left|\widehat{f}(\bsk)- \widehat{f}_N^{\mathrm{mult}}(\bsk)  \right| \le L \sum_{\bsk \notin\Acal_d}|\widehat{f}(\bsk)| \]
was established in \cite[Lemma~3.1]{KV19} (cf.~\cite[Lemma~15.8]{DKP22}). For the sake of completeness, and because a similar argument is necessary for the randomized setting discussed in the next subsection, we provide a proof below.

For any $\bsk\in \Acal_d$, substituting the Fourier expansion of $f$ into the definition of $\widehat{f}_N^{\mathrm{mult}}(\bsk)$ yields
\begin{align*}
    \widehat{f}_N^{\mathrm{mult}}(\bsk) & = \frac{1}{\xi(\bsk)}\sum_{t=1}^L \frac{U_{t}(\bsk)}{n_{t}}\sum_{i=0}^{n_t-1}f(\bsy^{(t)}_i) \, e^{-2\pi i\bsk\cdot \bsy^{(t)}_i}\\
    & = \frac{1}{\xi(\bsk)}\sum_{t=1}^L U_{t}(\bsk)\sum_{\bsh\in \ZZ^d}\widehat{f}(\bsh)\, \frac{1}{n_{t}}\sum_{i=0}^{n_t-1}e^{2\pi i(\bsh-\bsk)\cdot \bsy^{(t)}_i}\\
    & = \frac{1}{\xi(\bsk)}\sum_{t=1}^L U_{t}(\bsk)\sum_{\substack{\bsh\in \ZZ^d\\ \bsh-\bsk\in P_{n_t,\bsg_t}^{\perp}}}\widehat{f}(\bsh) \\
    & = \widehat{f}(\bsk)+\frac{1}{\xi(\bsk)}\sum_{t=1}^L U_{t}(\bsk)\sum_{\bsl\in P_{n_t,\bsg_t}^{\perp}\setminus \{\bszero\}}\widehat{f}(\bsk+\bsl),
\end{align*}
where the third equality follows from the character property of the lattice (Lemma~\ref{lem:character}), and the last equality follows from the change of variables $\bsl=\bsh-\bsk$. Consequently, we get
\begin{align*}
    \sum_{\bsk\in\Acal_d}\left|\widehat{f}(\bsk)- \widehat{f}_N^{\mathrm{mult}}(\bsk)  \right| & = \sum_{\bsk\in\Acal_d}\left|\frac{1}{\xi(\bsk)}\sum_{t=1}^L U_{t}(\bsk)\sum_{\bsl\in P_{n_t,\bsg_t}^{\perp}\setminus \{\bszero\}}\widehat{f}(\bsk+\bsl) \right|\\
    & \le \sum_{t=1}^L\sum_{\bsk\in\Acal_d} U_{t}(\bsk)\sum_{\bsl\in P_{n_t,\bsg_t}^{\perp}\setminus \{\bszero\}}|\widehat{f}(\bsk+\bsl) |\\
    & = \sum_{t=1}^L\sum_{\bsk\in \Acal_{d,t}} \sum_{\bsl\in P_{n_t,\bsg_t}^{\perp}\setminus \{\bszero\}}|\widehat{f}(\bsk+\bsl) |.
\end{align*}
By Lemma~\ref{lem:multiple_rank-1_property}, for each $t\in \{1,\ldots,L\}$, the sets of indices $\{\bsk + \bsl \mid \bsl \in P_{n_t, \bsg_t}^\perp \setminus \{\bszero\}\}$ for distinct $\bsk \in \Acal_{d,t}$ are pairwise disjoint and contained in $\ZZ^d \setminus \Acal_d$. Thus, the inner sums satisfy
\begin{align*}
\sum_{\bsk\in \Acal_{d,t}} \sum_{\bsl \in P_{n_t,\bsg_t}^{\perp} \setminus {\bszero}} |\widehat{f}(\bsk+\bsl)| \le \sum_{\bsk \notin \Acal_d} |\widehat{f}(\bsk)|.
\end{align*}
Summing over $t$, we obtain the bound $L \sum_{\bsk \notin \Acal_d} |\widehat{f}(\bsk)|$ as claimed.

Finally, by applying the Cauchy-Schwarz inequality and Lemma~\ref{lem:sum_weights_hyperbolic_cross} to the total error in \eqref{eq:inf_error_analytic}, we conclude
\begin{align*}
    \| f-A_N^{\mathrm{mult}}(f)\|_{L_{\infty}} & \le (L+1) \sum_{\bsk\notin \Acal_d}|\widehat{f}(\bsk) |\\
    & \le (L+1)\left( \sum_{\bsk\notin \Acal_d}r_{\alpha,\bsgamma}^2(\bsk)|\widehat{f}(\bsk) |^2 \right)^{1/2}\left( \sum_{\bsk\notin \Acal_d}\frac{1}{r_{\alpha,\bsgamma}^2(\bsk)} \right)^{1/2}\\
    & \le \|f\|_{d,\alpha,\bsgamma}\frac{L+1}{M^{1-\lambda/2}}\left(\frac{8(3-\lambda)}{2-\lambda}\sum_{u\subseteq [1{:}d]}\gamma_u^{\lambda}(2\zeta(\alpha\lambda))^{|u|}\right)^{1/2}.
\end{align*}
This completes the proof.
\end{proof}

By combining Theorem~\ref{thm:main_result_deterministic} with the bound on the number of sampling points given in Corollary~\ref{cor:number_sampling}, we obtain the convergence rate of the worst-case $L_{\infty}$ error in terms of the number of function evaluations $N$.

\begin{corollary} \label{cor:main_result_deterministic}
Under the same assumptions as in Theorem~\ref{thm:main_result_deterministic} and Proposition~\ref{prop:complexity}, for any $\lambda \in (1/\alpha, 2)$ and $\beta \in (0, 1)$, there exists a constant $C_{c,\delta,\beta,\lambda} > 0$ such that the worst-case $L_{\infty}$ error of the deterministic algorithm $A_N^{\mathrm{mult}}$ satisfies
\begin{align*}
\mathrm{err}_{\infty}(A_N^{\mathrm{mult}},H_{d,\alpha,\bsgamma}) \le \frac{C_{c,\delta,\beta,\lambda}}{N^{1/(\lambda(1+\beta))-1/2}}\left(\sum_{u\subseteq [1{:}d]}\gamma_u^{\lambda}(2\zeta(\alpha\lambda))^{|u|}\right)^{1/\lambda},
\end{align*}
with probability at least $1-\delta$. Here, the constant $C_{c,\delta,\beta,\lambda}$ depends only on $c, \delta, \beta, \lambda$ and is independent of $N, \alpha$, and the weights $\bsgamma$. In particular, as $\lambda \to 1/\alpha$ and $\beta \to 0$, the convergence rate gets arbitrarily close to $\mathcal{O}(N^{-\alpha + 1/2})$.
\end{corollary}

\begin{proof}
    From the construction in Algorithm~\ref{Algorithm_2}, there exists a constant $C'_{c,\delta}>0$, depending only on $c$ and $\delta$, such that the number of lattices $L$ satisfies
    \begin{align*}
        L+1 \le L_{\max}+1\le 2L_{\max} \le C'_{c,\delta}\log | \Acal_d| \le \frac{2C'_{c,\delta}}{\beta}| \Acal_d|^{\beta/2},
    \end{align*}
    where we used the elementary inequality $\log x\le x^{\beta}/\beta$, which holds for any $x>0$ and $\beta\in (0,1)$. By applying Lemma~\ref{lem:size_hyperbolic_cross} and Corollary~\ref{cor:number_sampling}, together with the above bound on $L+1$, to the worst-case $L_{\infty}$ error bound shown in Theorem~\ref{thm:main_result_deterministic}, we obtain
    \begin{align*}
        \mathrm{err}_{\infty}(A_N^{\mathrm{mult}},H_{d,\alpha,\bsgamma}) & \le \frac{L+1}{M^{1-\lambda/2}} \left(\frac{8(3-\lambda)}{2-\lambda}\sum_{u\subseteq [1{:}d]}\gamma_u^{\lambda}(2\zeta(\alpha\lambda))^{|u|}\right)^{1/2} \\
        & \le \frac{1}{M^{1-\lambda(1+\beta)/2}} \cdot \frac{2C'_{c,\delta}}{\beta}\left(\frac{8(3-\lambda)}{2-\lambda}\right)^{1/2} \left(\sum_{u\subseteq [1{:}d]}\gamma_u^{\lambda}(2\zeta(\alpha\lambda))^{|u|}\right)^{(1+\beta)/2} \\
        & \le \left( \frac{C_{c,\delta,\beta}^{1/(\lambda(1+\beta))}}{N^{1/(\lambda(1+\beta))}}\left(\sum_{u\subseteq [1{:}d]}\gamma_u^{\lambda}(2\zeta(\alpha\lambda))^{|u|}\right)^{1/\lambda}\right)^{1-\lambda(1+\beta)/2} \\
        & \quad \times \frac{2C_{c,\delta}}{\beta}\left(\frac{8(3-\lambda)}{2-\lambda}\right)^{1/2} \left(\sum_{u\subseteq [1{:}d]}\gamma_u^{\lambda}(2\zeta(\alpha\lambda))^{|u|}\right)^{(1+\beta)/2}\\
        & \le \frac{C_{c,\delta,\beta,\lambda}}{N^{1/(\lambda(1+\beta))-1/2}}\left(\sum_{u\subseteq [1{:}d]}\gamma_u^{\lambda}(2\zeta(\alpha\lambda))^{|u|}\right)^{1/\lambda},
    \end{align*}
    for any $\lambda\in (1/\alpha,2)$ and $\beta\in (0,1)$. Here, the constant $C_{c,\delta,\beta}>0$ is the same as in Corollary~\ref{cor:number_sampling}, and the constant $C_{c,\delta,\beta,\lambda}>0$ depends only on $c,\delta,\beta,\lambda$ and is independent of $N,\alpha$ and $\bsgamma$. This completes the proof.
\end{proof}

\begin{remark}
    The convergence rate of $\mathcal{O}(N^{-\alpha+1/2})$ is known to be the best possible, up to logarithmic factors, for function approximation in the $L_{\infty}$ norm within the Korobov space \cite{BDSU16}. Thus, this corollary shows that the multiple lattice-based algorithm can achieve a nearly optimal convergence rate for any fixed $\alpha>1/2$.
\end{remark}
\subsection{The worst-case root mean squared \texorpdfstring{$L_{2}$}{L-2} error}\label{subsec:L_2}

Let us move on to the analysis of the worst-case root mean squared $L_2$ error of the randomized multiple lattice-based algorithm $A_{N,\bsDelta}^{\mathrm{mult}}$, defined as
\[ \mathrm{err}_{2}^{\text{rand}}(A_{N,\bsDelta}^{\mathrm{mult}},H_{d,\alpha,\bsgamma}) \coloneqq \sup_{\substack{f\in H_{d,\alpha,\bsgamma}\\\| f\|_{d,\alpha,\bsgamma}\le 1}}\\ \left(\int_{[0,1]^d} \|f-A_{N,\bsDelta}^{\mathrm{mult}}(f)\|_{L_2}^2 \rd \bsDelta\right)^{1/2}, \]
where the squared $L_{2}$ error for an individual function and a fixed shift $\bsDelta$ is given by
\[ \|f-A_{N,\bsDelta}^{\mathrm{mult}}(f)\|^2_{L_2}=\int_{[0,1]^d}\left|f(\bsx)-A_{N,\bsDelta}^{\mathrm{mult}}(f)(\bsx)\right|^2\rd \bsx.\]
We show an upper bound on this worst-case error for any fixed smoothness $\alpha>1/2$ and general weights $\bsgamma$ as follows.

\begin{theorem}\label{thm:main_result_randomized}
    For $\alpha > 1/2$, $d \in \NN$, a collection of non-negative weights $\bsgamma=(\gamma_u)_{u\subset \NN, |u|<\infty}$, and $M \ge 1$, let $\Acal_d$ be defined as in \eqref{eq:weight_hyper}. Consider the randomized algorithm $A_{N,\bsDelta}^{\mathrm{mult}}$ of the form \eqref{eq:randomized_algorithm}, where the number of lattices $L$ and the lattice parameters $(n_t,\bsg_t)$ for $t=1,\ldots,L$ are obtained by Algorithm~\ref{Algorithm_2}. Provided that the parameters $c\in (1,\infty)$ and $\delta\in (0,1)$ are chosen such that the assumptions in Proposition~\ref{prop:complexity} are satisfied, then we have
    \[ \mathrm{err}_{2}^{\text{rand}}(A_{N,\bsDelta}^{\mathrm{mult}},H_{d,\alpha,\bsgamma}) \le \frac{\sqrt{L+1}}{M}, \]
    with probability $1-\delta$ (with respect to the construction of the lattices).
\end{theorem}

\begin{proof}
It suffices to show that for any individual function $f\in H_{d,\alpha,\bsgamma}$, the inequality
\[ \int_{[0,1]^d} \|f-A_{N,\bsDelta}^{\mathrm{mult}}(f)\|_{L_2}^2 \rd \bsDelta \le \|f\|_{d,\alpha,\bsgamma}^2\frac{L+1}{M^2} \]
holds. Using the Fourier series expansion of $f$ and the definition of the randomized algorithm \eqref{eq:randomized_algorithm}, the squared $L_{2}$ error for a fixed shift $\bsDelta\in [0,1]^d$ is given by
\begin{align}
\|f-A_{N,\bsDelta}^{\mathrm{mult}}(f)\|^2_{L_2} & = \int_{[0,1]^d}\left|\sum_{\bsk\notin\Acal_d} \widehat{f}(\bsk) \, e^{2\pi i\bsk\cdot \bsx}+\sum_{\bsk\in\Acal_d} \left(\widehat{f}(\bsk)-\widehat{f}_{N,\bsDelta}^{\mathrm{mult}}(\bsk)\right)  e^{2\pi i\bsk\cdot \bsx}\right|^2 \rd \bsx \notag \\
& = \sum_{\bsk \notin\Acal_d}|\widehat{f}(\bsk)|^2+ \sum_{\bsk\in\Acal_d}\left|\widehat{f}(\bsk)- \widehat{f}_{N,\bsDelta}^{\mathrm{mult}}(\bsk)  \right|^2 , \label{eq:L2_error_analytic}
\end{align}
where the second equality follows from Parseval's identity.

Similarly to the deterministic case, we aim to show the following bound
\[ \int_{[0,1]^d} \sum_{\bsk\in\Acal_d}\left|\widehat{f}(\bsk)- \widehat{f}_{N,\bsDelta}^{\mathrm{mult}}(\bsk)  \right|^2 \rd \bsDelta \le L \sum_{\bsk \notin\Acal_d}|\widehat{f}(\bsk)|^2. \]
For any $\bsk\in \Acal_d$, substituting the Fourier expansion of $f$ into the definition of $\widehat{f}_{N,\bsDelta}^{\mathrm{mult}}(\bsk)$ yields
\begin{align*}
    \widehat{f}_{N,\bsDelta}^{\mathrm{mult}}(\bsk) & = \frac{1}{\xi(\bsk)}\sum_{t=1}^L \frac{U_{t}(\bsk)}{n_{t}}\sum_{i=0}^{n_t-1}f(\bsy^{(t)}_{i,\bsDelta}) \, e^{-2\pi i\bsk\cdot \bsy^{(t)}_{i,\bsDelta}}\\
    & = \frac{1}{\xi(\bsk)}\sum_{t=1}^L U_{t}(\bsk)\sum_{\bsh\in \ZZ^d}\widehat{f}(\bsh)\, \frac{1}{n_{t}}\sum_{i=0}^{n_t-1}e^{2\pi i(\bsh-\bsk)\cdot (\bsy^{(t)}_{i}+\bsDelta)}\\
    & = \frac{1}{\xi(\bsk)}\sum_{t=1}^L U_{t}(\bsk)\sum_{\substack{\bsh\in \ZZ^d\\ \bsh-\bsk\in P_{n_t,\bsg_t}^{\perp}}}\widehat{f}(\bsh)\, e^{2\pi i(\bsh-\bsk)\cdot \bsDelta} \\
    & = \widehat{f}(\bsk)+\frac{1}{\xi(\bsk)}\sum_{t=1}^L U_{t}(\bsk)\sum_{\bsl\in P_{n_t,\bsg_t}^{\perp}\setminus \{\bszero\}}\widehat{f}(\bsk+\bsl)\, e^{2\pi i\bsl \cdot \bsDelta},
\end{align*}
where the third equality follows from the character property of the lattice (Lemma~\ref{lem:character}), and the last equality follows from the change of variables $\bsl=\bsh-\bsk$. By the Cauchy--Schwarz inequality and the fact that $U^2_t(\bsk) = U_t(\bsk)$, noting that the sum over $t$ contains exactly $\xi(\bsk)$ non-zero terms for each $\bsk\in \Acal_d$, we have
\begin{align*}
    & \int_{[0,1]^d} \sum_{\bsk\in\Acal_d}\left|\widehat{f}(\bsk)- \widehat{f}_{N,\bsDelta}^{\mathrm{mult}}(\bsk)  \right|^2 \rd \bsDelta \\
    & = \int_{[0,1]^d} \sum_{\bsk\in\Acal_d}\left|\frac{1}{\xi(\bsk)}\sum_{t=1}^L U_{t}(\bsk)\sum_{\bsl\in P_{n_t,\bsg_t}^{\perp}\setminus \{\bszero\}}\widehat{f}(\bsk+\bsl)\, e^{2\pi i\bsl \cdot \bsDelta}  \right|^2 \rd \bsDelta\\
    & \le \int_{[0,1]^d} \sum_{\bsk\in\Acal_d}\frac{\xi(\bsk)}{(\xi(\bsk))^2} \sum_{t=1}^L (U_{t}(\bsk))^2 \left|\sum_{\bsl\in P_{n_t,\bsg_t}^{\perp}\setminus \{\bszero\}}\widehat{f}(\bsk+\bsl)\, e^{2\pi i\bsl \cdot \bsDelta}  \right|^2 \rd \bsDelta\\
    & = \sum_{\bsk\in\Acal_d}\frac{1}{\xi(\bsk)} \sum_{t=1}^L (U_{t}(\bsk))^2 \int_{[0,1]^d} \left|\sum_{\bsl\in P_{n_t,\bsg_t}^{\perp}\setminus \{\bszero\}}\widehat{f}(\bsk+\bsl)\, e^{2\pi i\bsl \cdot \bsDelta}  \right|^2 \rd \bsDelta \\
    & = \sum_{\bsk\in\Acal_d}\frac{1}{\xi(\bsk)} \sum_{t=1}^L (U_{t}(\bsk))^2 \sum_{\bsl\in P_{n_t,\bsg_t}^{\perp}\setminus \{\bszero\}}|\widehat{f}(\bsk+\bsl)|^2 \\
    & \le \sum_{t=1}^L\sum_{\bsk\in\Acal_d} U_{t}(\bsk)\sum_{\bsl\in P_{n_t,\bsg_t}^{\perp}\setminus \{\bszero\}}|\widehat{f}(\bsk+\bsl)|^2\\
    & = \sum_{t=1}^L\sum_{\bsk\in \Acal_{d,t}} \sum_{\bsl\in P_{n_t,\bsg_t}^{\perp}\setminus \{\bszero\}}|\widehat{f}(\bsk+\bsl) |^2.
\end{align*}
Applying Lemma~\ref{lem:multiple_rank-1_property}, the sets of indices $\{\bsk + \bsl \mid \bsl \in P_{n_t, \bsg_t}^\perp \setminus \{\bszero\}\}$ for distinct $\bsk \in \Acal_{d,t}$ are pairwise disjoint and contained in $\ZZ^d \setminus \Acal_d$. Thus, 
\begin{align*}
\sum_{\bsk\in \Acal_{d,t}} \sum_{\bsl\in P_{n_t,\bsg_t}^{\perp}\setminus \{\bszero\}}|\widehat{f}(\bsk+\bsl) |^2 \le \sum_{\bsk \notin \Acal_d} |\widehat{f}(\bsk)|^2.
\end{align*}
Summing over $t$, we obtain the bound $L \sum_{\bsk \notin \Acal_d} |\widehat{f}(\bsk)|^2$ as claimed.

Altogether, the mean-squared $L_2$ error is bounded by
\begin{align*}
    \int_{[0,1]^d} \|f-A_{N,\bsDelta}^{\mathrm{mult}}(f)\|_{L_2}^2 \rd \bsDelta & = \sum_{\bsk \notin\Acal_d}|\widehat{f}(\bsk)|^2+ \int_{[0,1]^d} \sum_{\bsk\in\Acal_d}\left|\widehat{f}(\bsk)- \widehat{f}_{N,\bsDelta}^{\mathrm{mult}}(\bsk)  \right|^2 \rd \bsDelta\\
    & \le (L+1) \sum_{\bsk \notin\Acal_d}|\widehat{f}(\bsk)|^2\\
    & \le (L+1) \left(\sum_{\bsk \notin\Acal_d}r^2_{\alpha,\bsgamma}(\bsk)|\widehat{f}(\bsk)|^2\right)\sup_{\bsk \notin\Acal_d}\frac{1}{r^2_{\alpha,\bsgamma}(\bsk)}\\
    & \le \|f\|^2_{d,\alpha,\bsgamma}\frac{L+1}{M^2},
\end{align*}
where we used $r_{\alpha,\bsgamma}(\bsk) > M$ for $\bsk\notin \Acal_d$. Taking the square root and the supremum over $\|f\|_{d,\alpha,\bsgamma}\le 1$ completes the proof.
\end{proof}

By combining Theorem~\ref{thm:main_result_randomized} with the bound on the number of sampling points given in Corollary~\ref{cor:number_sampling}, we obtain the convergence rate of the worst-case root mean squared $L_2$ error in terms of the number of function evaluations $N$. As the result can be proven in an almost identical way to Corollary~\ref{cor:main_result_deterministic}, we omit the proof.

\begin{corollary} \label{cor:main_result_randomized}
Under the same assumptions as in Theorem~\ref{thm:main_result_randomized} and Proposition~\ref{prop:complexity}, for any $\lambda \in (1/\alpha, 2)$ and $\beta \in (0, 1)$, there exists a constant $C'_{c,\delta,\beta,\lambda} > 0$ such that the worst-case root mean squared $L_2$ error of the randomized algorithm $A_{N,\bsDelta}^{\mathrm{mult}}$ satisfies
\begin{align*}
\mathrm{err}_{2}^{\text{rand}}(A_{N,\bsDelta}^{\mathrm{mult}},H_{d,\alpha,\bsgamma}) \le \frac{C'_{c,\delta,\beta,\lambda}}{N^{(1-\lambda\beta)/(\lambda(1+\beta))}}\left(\sum_{u\subseteq [1{:}d]}\gamma_u^{\lambda}(2\zeta(\alpha\lambda))^{|u|}\right)^{1/\lambda},
\end{align*}
with probability at least $1-\delta$. Here, the constant $C'_{c,\delta,\beta,\lambda}$ depends only on $c, \delta, \beta, \lambda$ and is independent of $N, \alpha$, and the weights $\bsgamma$. In particular, as $\lambda \to 1/\alpha$ and $\beta \to 0$, the convergence rate gets arbitrarily close to $\mathcal{O}(N^{-\alpha})$.
\end{corollary}

\begin{remark}
    The convergence rate of $\mathcal{O}(N^{-\alpha})$ has been shown to be the best possible, up to logarithmic factors, for function approximation in the $L_2$ norm within the Korobov space both in the deterministic and randomized settings; see, for instance, \cite{WW07}. Thus, this corollary shows that our randomized multiple lattice-based algorithm can achieve a nearly optimal convergence rate for any fixed $\alpha>1/2$.
\end{remark}

\begin{remark}
    Here, we consider the worst-case root mean squared $L_2$ error by incorporating a random shift. It is natural to ask whether similar optimal convergence rates and good tractability results, the latter of which shall be shown in the next subsection, can be achieved for the deterministic $L_2$ error without such a shift. As shown in the proof of Theorem~\ref{thm:main_result_randomized}, for an individual function $f\in H_{d,\alpha,\bsgamma}$, the squared $L_2$ error for the randomized algorithm satisfies
    \[ \int_{[0,1]^d} \|f-A_{N,\bsDelta}^{\mathrm{mult}}(f)\|_{L_2}^2 \rd \bsDelta\le \sum_{\bsk \notin\Acal_d}|\widehat{f}(\bsk)|^2 + \sum_{\bsk\in\Acal_d} \sum_{t=1}^L U_{t}(\bsk)\sum_{\bsl\in P_{n_t,\bsg_t}^{\perp}\setminus \{\bszero\}}|\widehat{f}(\bsk+\bsl)|^2, \]
    whereas for the deterministic algorithm $A_{N}^{\mathrm{mult}}$, we have
    \[ \|f-A_{N}^{\mathrm{mult}}(f)\|_{L_2}^2 = \sum_{\bsk \notin\Acal_d}|\widehat{f}(\bsk)|^2 + \sum_{\bsk\in\Acal_d} \left|\frac{1}{\xi(\bsk)}\sum_{t=1}^L U_{t}(\bsk)\sum_{\bsl\in P_{n_t,\bsg_t}^{\perp}\setminus \{\bszero\}}\widehat{f}(\bsk+\bsl)\right|^2. \]
    A crucial difference here lies in whether the square is taken inside or outside the sum over the points $\bsl$ in the dual lattice, i.e., the integration over $\bsDelta$ effectively eliminates the cross-terms in the randomized case. It remains an open question whether nearly optimal convergence rates can be attained along with strong polynomial tractability for the deterministic $L_2$ error under some summability conditions on the weights.
\end{remark}

\subsection{Tractability}\label{subsec:tractability}

Finally, we discuss the tractability of the multivariate approximation problem in the weighted Korobov space $H_{d,\alpha,\bsgamma}$. Since Corollaries \ref{cor:main_result_deterministic} and \ref{cor:main_result_randomized} both rely on Proposition~\ref{prop:complexity}, and respectively on Theorems~\ref{thm:main_result_deterministic} and \ref{thm:main_result_randomized}, the following argument assumes that $M\ge 1$ and $|\Acal_d|\ge 2$. Note that $M$ and $|\Acal_d|$ are closely related through \eqref{eq:weight_hyper}. In fact, we require a stronger condition on $M$, i.e., we choose $M\ge \max(1,\min_u \gamma_u^{-1})$, which ensures $|\Acal_d|\ge 2$ since the corresponding $\Acal_d$ contains $\bszero$ and at least one non-zero vector. Furthermore, we recall that the error bounds in Corollaries \ref{cor:main_result_deterministic} and \ref{cor:main_result_randomized} hold under the assumption that
\[ \eta\coloneqq c(|\Acal_d|-1)\geq \max(N_{\Acal_d},4L_{\max}\log L_{\max}), \]
where $N_{\Acal_d}$ is defined as in \eqref{hyper_extension_length}. To discuss the tractability for multivariate problems in high dimensions, it is essential to check when this condition is satisfied.

When the index set $\Acal_d$ is defined as in \eqref{eq:weight_hyper}, the parameter $N_{\Acal_d}$ satisfies
\begin{align*}
    N_{\Acal_d} & = \max_{j=1,\ldots,d}\left(\max_{\substack{\bsk\in \ZZ^d\\ r_{\alpha,\bsgamma}(\bsk)\le M}} k_j-\min_{\substack{\bsh\in \ZZ^d\\ r_{\alpha,\bsgamma}(\bsh)\le M}} h_j  \right)\\
    & = \max_{j=1,\ldots,d}\left(\max_{j\in u\subseteq [1{:}d]}\max_{\substack{\bsk_u\in (\ZZ\setminus \{0\})^{|u|}\\ r_{\alpha,\bsgamma}(\bsk_u,\bszero)\le M}} k_j-\min_{j\in u\subseteq [1{:}d]}\min_{\substack{\bsh_u\in (\ZZ\setminus \{0\})^{|u|}\\ r_{\alpha,\bsgamma}(\bsh_u,\bszero)\le M}} h_j  \right)\\
    & = 2\max_{j=1,\ldots,d}\max_{j\in u\subseteq [1{:}d]}\lfloor (\gamma_u M)^{1/\alpha}\rfloor.
\end{align*}
Regarding the size of $\Acal_d$, we observe that
\begin{align*}
    |\Acal_d| -1 & = \sum_{\bsk\in \ZZ^d\setminus \{\bszero\}}\mathds{1}_{r_{\alpha,\bsgamma} (\bsk)\leq M} = \sum_{\emptyset \ne u\subseteq [1{:}d]}\sum_{\bsk_u\in (\ZZ\setminus \{0\})^{|u|}}\mathds{1}_{r_{\alpha,\bsgamma}(\bsk_u,\bszero)\le M}\\
    & \ge \sum_{\emptyset \ne u\subseteq [1{:}d]}2^{|u|}\cdot \lfloor (\gamma_u M)^{1/\alpha}\rfloor\\
    & \ge 2\max_{j=1,\ldots,d}\max_{j\in u\subseteq [1{:}d]}\lfloor (\gamma_u M)^{1/\alpha}\rfloor = N_{\Acal_d},
\end{align*}
where the first inequality follows from the fact that the inner sum over $\bsk_u$ can be bounded below by considering a subset of indices such that, for a fixed $j\in u$, we set all components $k_i\in \{-1,1\}$ for $i\in u\setminus \{j\}$. Then, the condition $r_{\alpha,\bsgamma}(\bsk_u,\bszero)\le M$ is satisfied for any $k_j\in \{-\lfloor (\gamma_u M)^{1/\alpha}\rfloor,\ldots,-1,1,\ldots,\lfloor (\gamma_u M)^{1/\alpha}\rfloor\}$, yielding the derived lower estimate.
This bound implies that the assumption $\eta\ge N_{\Acal_d}$ holds for any $c>1$.

Next, to show when the assumption $\eta\ge 4L_{\max}\log L_{\max}$ holds, let us consider the case $\delta=1/2$. Recall that $|\Acal_d|\ge 2$. From the definition of $L_{\max}$, by applying the inequality $\log (x+1)\le x^{\beta}/\beta$, which holds for any $x\ge 0$ and $\beta\in (0,1)$ (here we take $\beta=1/3$), we have
\begin{align*}
    4L_{\max}\log L_{\max} & \le 12L_{\max}^{4/3} \le 12 \left( \left(\frac{c}{c-1} \right)^2 \frac{\log | \Acal_d| +\log 2}{2}+1\right)^{4/3}\\
    & \le 12\left(\frac{c}{c-1} \right)^{8/3}\left(1+\frac{1}{\log 2}\right)^{4/3}\left( \log |\Acal_d|\right)^{4/3}\\
    & \le 36\sqrt[3]{3}\left(\frac{c}{c-1} \right)^{8/3}\left(1+\frac{1}{\log 2}\right)^{4/3}\left( |\Acal_d|-1\right)^{4/9}.
\end{align*}
Comparing this with $\eta= c(|\Acal_d|-1)$, the condition $\eta\ge 4L_{\max}\log L_{\max}$ is satisfied if
\[ c\ge 36\sqrt[3]{3}\left(\frac{c}{c-1} \right)^{8/3}\left(1+\frac{1}{\log 2}\right)^{4/3}.\]
Numerically, this inequality is satisfied for $c\ge 173.46$. It is important to note that the assumptions required for the error bounds in Corollaries~\ref{cor:main_result_deterministic} and \ref{cor:main_result_randomized} can be satisfied with $\delta=1/2$ and a constant $c\ge 173.46$ independently of the dimension $d$, smoothness $\alpha$, and the weight parameters $\bsgamma$.

We now state the conditions for \emph{strong polynomial tractability}. Following the standard framework in information-based complexity \cite{NW08,NW10,NW12}, let $N(\varepsilon,d)$ denote the information complexity, which is defined as the minimal number of function evaluations required to reduce the worst-case error to at most $\varepsilon\in (0,1)$ for a $d$-dimensional problem. A multivariate problem is said to be \emph{strongly polynomially tractable} if there exist absolute constants $C>0$ and $p\ge 0$ such that
\[ N(\varepsilon,d)\le C\varepsilon^{-p}\quad \text{for all $\varepsilon\in (0,1)$ and $d\in \NN$}.\]
In other words, the information complexity is bounded by a polynomial in $\varepsilon^{-1}$ that is completely independent of the dimension $d$. Since we have provided a constructive algorithm that achieves a dimension-independent error bound, the strong polynomial tractability of our problem directly follows. The next result is thus an immediate consequence of Corollaries~\ref{cor:main_result_deterministic} and \ref{cor:main_result_randomized}.

\begin{theorem}
    For $\alpha > 1/2$, $d \in \NN$, a collection of non-negative weights $\bsgamma=(\gamma_u)_{u\subset \NN, |u|<\infty}$, a multivariate approximation in the weighted Korobov space $H_{d,\alpha,\bsgamma}$ is strongly polynomially tractable in both the deterministic $L_{\infty}$ sense and the randomized $L_2$ sense if there exists $\lambda\in (1/\alpha, 2)$ such that
    \[ \sum_{\substack{u\subset \NN\\ |u|<\infty}}\gamma_u^{\lambda}(2\zeta(\alpha\lambda))^{|u|} < \infty. \]
\end{theorem}

When the weights $\bsgamma$ have some specific structure, in particular, when they are either product weights or POD weights, the condition given in the above theorem can be simplified as follows. We refer to the recent work \cite{P25} for more refined analyses for the POD and SPOD weights cases.

\begin{corollary}
    Under the same assumptions as in the previous theorem:
    \begin{enumerate}
        \item If $\bsgamma$ are product weights, i.e., $\gamma_u=\prod_{j\in u}\gamma_j$, strong polynomial tractability holds if there exists $\lambda\in (1/\alpha,2)$ such that
        \[ \sum_{j=1}^{\infty}\gamma_j^{\lambda}<\infty.\]
        \item If $\bsgamma$ are POD weights, i.e., $\gamma_u=\Gamma_{|u|}\prod_{j\in u}\gamma_j$, strong polynomial tractability holds if there exist $c>0$ and $\lambda\in (1/\alpha,2)$ such that
        \[ 2\zeta(\alpha\lambda)\sum_{j=1}^{\infty}\gamma_j^{\lambda}< 1\quad \text{and}\quad \Gamma_{\ell}\le c\,(\ell!)^{1/\lambda}\quad \text{for all $\ell\in \NN$}.\]
    \end{enumerate}
\end{corollary}

\begin{proof}
    For the first assertion, if the claimed condition holds, we have
    \begin{align*}
        \sum_{\substack{u\subset \NN\\ |u|<\infty}}\gamma_u^{\lambda}(2\zeta(\alpha\lambda))^{|u|} & = \sum_{\substack{u\subset \NN\\ |u|<\infty}}\prod_{j\in u}\left(2\zeta(\alpha\lambda)\gamma_j^{\lambda}\right) = \prod_{j=1}^{\infty}\left(1+2\zeta(\alpha\lambda)\gamma_j^{\lambda}\right)\\
        & \le \exp\left( 2\zeta(\alpha\lambda)\sum_{j=1}^{\infty}\gamma_j^{\lambda} \right)<\infty.
    \end{align*}
    This completes the proof of the first assertion.

    Regarding the second assertion, we observe that
    \begin{align*}
        \sum_{\substack{u\subset \NN\\ |u|<\infty}}\gamma_u^{\lambda}(2\zeta(\alpha\lambda))^{|u|} & = \sum_{\substack{u\subset \NN\\ |u|<\infty}}\Gamma_{|u|}^{\lambda}\prod_{j\in u}\left(2\zeta(\alpha\lambda)\gamma_j^{\lambda}\right)\\
        & \le \sum_{\substack{u\subset \NN\\ |u|<\infty}}c^{\lambda}|u|!\prod_{j\in u}\left(2\zeta(\alpha\lambda)\gamma_j^{\lambda}\right)\\
        & = c^{\lambda}\sum_{k=0}^{\infty} k! \sum_{\substack{u\subset \NN\\ |u|=k}}\prod_{j\in u}\left(2\zeta(\alpha\lambda)\gamma_j^{\lambda}\right)\\
        & \le c^{\lambda}\sum_{k=0}^{\infty}\left( 2\zeta(\alpha\lambda)\sum_{j=1}^{\infty}\gamma_j^{\lambda}\right)^{k} \\
        & = c^{\lambda}\left( 1-2\zeta(\alpha\lambda)\sum_{j=1}^{\infty}\gamma_j^{\lambda}\right)^{-1}<\infty,
    \end{align*}
    where the second inequality follows from the property that, for any summable sequence of non-negative $x_j$, $$\sum_{\substack{u\subset \NN \\ |u|=k}} \prod_{j \in u} x_j \le \frac{1}{k!} \left( \sum_{j=1}^\infty x_j \right)^k,$$
    cf.~\cite[Lemma~6.3]{KSS12} and \cite[Lemma~19]{GKNSSS15}.
    This completes the proof.
\end{proof}

\begin{remark}
    In \cite[Theorem~15.11]{DKP22}, it is shown that, for product weights, strong polynomial tractability holds if there exists $\tau>1$ such that
    \[ \sum_{j=1}^{\infty}\gamma_j^{\min(1,\tau/2)}<\infty,\]
    achieving a rate arbitrarily close to $\mathcal{O}(N^{-(\alpha-1/2)/\tau})$ in the deterministic worst-case $L_{\infty}$ sense. To achieve a nearly optimal rate $\mathcal{O}(N^{-\alpha+1/2})$, one requires $\tau$ to be very close to $1$, which implies a summability condition with exponent $\min(1, \tau/2) \approx 1/2$. In contrast, the above corollary shows that for any fixed $\alpha > 1/2$, the condition $\sum_{j=1}^\infty \gamma_j^\lambda < \infty$ for some $\lambda \in (1/\alpha, 2)$ is sufficient, with a corresponding rate $\mathcal{O}(N^{-1/\lambda+1/2})$. Comparing these two conditions, we see that for the same target rate $\mathcal{O}(N^{-\alpha+1/2})$, our summability requirement is weaker than that of \cite{DKP22} when $1/2 < \alpha < 2$ (since $1/\alpha > 1/2$), while the opposite holds when $\alpha > 2$.

    This discrepancy stems from the differences in the choice of the index set $\Acal_d$ and the corresponding proof techniques. Our index set $\Acal_d$ is \emph{$\alpha$-dependent} as in \eqref{eq:weight_hyper}, whereas the index set in \cite[Theorem~15.11]{DKP22} is \emph{$\alpha$-independent}. As mentioned in Remark~\ref{rem:comparison_index_set}, the square roots of the weight parameters, $\bsgamma^{1/2}$, are considered in constructing their index set, which plays an important role in deriving the above $\alpha$-independent summability condition.
\end{remark}
\appendix

\section{Probabilistic construction of multiple rank-1 lattices}\label{appendix:a}
Following \cite{K19} and \cite[Chapter 15]{DKP22}, we describe a probabilistic algorithm to find suitable generating vectors for multiple rank-1 lattices. For a given weighted hyperbolic cross $\Acal_d$, the algorithm selects a sequence of $L$ rank-1 lattice point sets $P_{n_t, \bsg_t}$ for $t=1, \ldots, L$ in a greedy manner. The sizes $n_t$ are chosen as distinct prime numbers from a set of candidate primes $\Pcal_{\Acal_d}$ defined by
\begin{align*}
\Pcal_{\Acal_d} \coloneqq \left\{ n\in \NN \mid \text{$n$ is prime and $\left|\left\{\bsk\bmod {n} \mid \bsk\in \Acal_d\right\}\right|=|\Acal_d|$} \right\},
\end{align*}
where the modulo operation $\bsk\bmod n$ is applied componentwise. This set is guaranteed to contain all prime numbers greater than the maximum componentwise span:
\begin{align}\label{hyper_extension_length}
N_{\Acal_d}\coloneqq\max_{j=1,\ldots,d}\left(\max_{\bsk\in \Acal_d} k_j-\min_{\bsh\in \Acal_d} h_j  \right),
\end{align}
as noted in \cite[Remark~15.1]{DKP22}. For given parameters $\eta,L_{\max}\in \NN$, we consider the subset of $\Pcal_{\Acal_d,\eta,L_{\max}}\subset \Pcal_{\Acal_d}$ consisting of the $L_{\max}$ smallest primes in $\Pcal_{\Acal_d}$ that are greater than $\eta$.

\begin{algorithm}[t]
\caption{Probabilistic construction of multiple rank-1 lattices.}
\label{Algorithm_2}
\begin{algorithmic}[1]
\REQUIRE Hyperbolic cross $\Acal_d$, parameters $c \in (1, \infty)$ and $\delta \in (0, 1)$.
\STATE Set the maximum number of lattices:
\begin{align*}
L_{\max} \coloneqq \left\lceil \left( \frac{c}{c-1} \right)^2 \frac{\log | \Acal_d| - \log \delta}{2} \right\rceil.
\end{align*}
\STATE Set the threshold $\eta \coloneqq (|\Acal_d| - 1)c$.
\STATE Determine the set $\Pcal_{\Acal_d,\eta,L_{\max}}=\{p_1,\ldots,p_{L_{\max}}\}$ with $p_1<p_2<\dots<p_{L_{\max}}$.
\STATE Initialize $t \leftarrow 0$ and the set of covered frequencies $\mathcal{B}_d \leftarrow \varnothing$.
\WHILE{$|\mathcal{B}_d| < |\Acal_d|$ and $t < L_{\max}$}
\STATE $t \leftarrow t + 1$.
\STATE Set $n_t = p_t \in \Pcal_{\Acal_d,\eta,L_{\max}}$.
\STATE Choose $\bsg_t \in \{1, \ldots, n_t\}^d$ uniformly at random.
\STATE Identify the aliasing-free frequencies $\Acal_{d,t} \subseteq \Acal_d$ for $(n_t, \bsg_t)$:
\begin{align*}
\Acal_{d,t} = \left\{\bsk \in \Acal_d \mid U_t(\bsk)=1 \right\} .
\end{align*}
\IF{$\Acal_{d,t} \not\subseteq \mathcal{B}_d$}
\STATE Update the covered set: $\mathcal{B}_d \leftarrow \mathcal{B}_d \cup \Acal_{d,t}$.
\ELSE
\STATE $t \leftarrow t - 1$.
\ENDIF
\ENDWHILE
\RETURN $\{ (n_1, \bsg_1), \ldots, (n_L, \bsg_L) \}$ with $L=t$.
\end{algorithmic}
\end{algorithm}

The greedy construction summarized in Algorithm~\ref{Algorithm_2} probabilistically ensures that every frequency in $\Acal_d$ is covered by at least one aliasing-free lattice. The following proposition provides an upper bound on the total number of points $N \coloneqq \sum_{t=1}^L n_t$. We refer to \cite[Corollary~3.6]{K19} and \cite[Proposition~15.4]{DKP22} for the proof.

\begin{proposition}\label{prop:complexity}
Assume that $|\Acal_d|\geq 2$ , $c\in(1,\infty)$, and let $\eta\coloneqq c(|\Acal_d|-1)$. If $\eta$ satisfies 
\[ \eta\geq \max(N_{\Acal_d},4L_{\max}\log L_{\max}), \]
where $N_{\Acal_d}$ is defined as in \eqref{hyper_extension_length}, then Algorithm~\ref{Algorithm_2} covers the entire set $\Acal_d$ with probability at least $1-\delta$. Furthermore, the total number of sampling points $N=\sum_{t=1}^L n_t$ is bounded above by
\begin{align*}
N\leq 2cL_{\max}(|\Acal_d|-1).
\end{align*}
\end{proposition}

By leveraging the size estimate for the weighted hyperbolic cross from Lemma~\ref{lem:size_hyperbolic_cross}, we obtain an explicit bound for $N$ in terms of the radius $M$.

\begin{corollary}\label{cor:number_sampling}
    For $\alpha > 1/2$, $d \in \NN$, a collection of non-negative weights $\bsgamma=(\gamma_u)_{u\subset \NN, |u|<\infty}$, and $M > 0$, let $\Acal_d$ be defined as in \eqref{eq:weight_hyper}. Under the same assumptions as in Proposition~\ref{prop:complexity}, for any $\lambda>1/\alpha$ and $\beta\in (0,1)$, there exists a constant $C_{c,\delta,\beta}>0$ such that
    \[ N \le C_{c,\delta,\beta}M^{\lambda(1+\beta)} \left(\sum_{u\subseteq [1{:}d]}\gamma_u^{\lambda}(2\zeta(\alpha\lambda))^{|u|}\right)^{1+\beta}.\]
    Here, the constant $C_{c,\delta,\beta}$ depends only on the parameters $c, \delta, \beta$ and satisfies $C_{c,\delta,\beta} \to \infty$ as either $c \to \infty$, $\delta \to 0$, or $\beta \to 0$.
\end{corollary}

\begin{proof}
    From Proposition~\ref{prop:complexity} and the definition of $L_{\max}$ in Algorithm~\ref{Algorithm_2}, we have
    \begin{align*}
        N & \le 2cL_{\max}(|\Acal_d|-1) \\
        & \le 2c\left(\left( \frac{c}{c-1} \right)^2 \frac{\log | \Acal_d| - \log \delta}{2} +1 \right)(|\Acal_d|-1)\\
        & \le C_{c,\delta}|\Acal_d|\log | \Acal_d|,
    \end{align*}
    where $C_{c,\delta}>0$ is a constant depending only on $c$ and $\delta$. By applying the elementary inequality $\log x\le x^{\beta}/\beta$, which holds for any $x>0$ and $\beta\in (0,1)$, together with Lemma~\ref{lem:size_hyperbolic_cross}, we obtain
    \[ N \le \frac{C_{c,\delta}}{\beta}|\Acal_d|^{1+\beta}\le \frac{C_{c,\delta}}{\beta} M^{\lambda(1+\beta)} \left(\sum_{u\subseteq [1{:}d]}\gamma_u^{\lambda}(2\zeta(\alpha\lambda))^{|u|}\right)^{1+\beta},\]
    which completes the proof.
\end{proof}

\section{Lower bounds for single lattice-based algorithms}\label{appendix:b}

In this appendix, we provide two approaches to derive lower bounds on the worst-case $L_2$ error for single lattice-based algorithms. For simplicity, we consider the two-dimensional unweighted Korobov space with smoothness $\alpha>1/2$, and the number of points $N$ is assumed to be a prime number. The essence of both approaches lies in identifying a short non-zero vector in the dual lattice (which aliases to the origin) to construct a fooling function such that the corresponding approximation error is as large as possible. The first approach, relying on the pigeonhole principle, is based on \cite[Theorem~9.15]{HW81}. This result can be traced back to the early works of Smolyak \cite{Smo60} and Korobov \cite{K63}, and was recently exploited in \cite{BKUV17} (without referring to earlier works). The second approach, which is constructive and relies on the properties of continued fractions, follows the argument in \cite[Theorem~9.14]{HW81}, though we present a simplified derivation for clarity.

\subsection{A lower bound by Smolyak and Korobov}
Consider a rank-1 lattice $P_{N,\bsg}$ with generating vector $\bsg=(g_1,g_2)\in \{1,\ldots,N-1\}^2$. The elements $\bsk\in \ZZ^2$ of the corresponding dual lattice satisfy the congruence:
\[ k_1 g_1+k_2g_2\equiv 0 \pmod N.\]
By the pigeonhole principle, this equation is guaranteed to have at least one solution $\bsh$ such that
\[ 0< |h_1|,|h_2| \le \sqrt{N}. \]
Indeed, since there are $(\lfloor \sqrt{N}\rfloor+1)^2>N$ integer vectors $(k_1,k_2)$ in the set $\{0,\ldots,\lfloor \sqrt{N}\rfloor\}^2$, there must exist two distinct vectors $\bsk'$ and $\bsk''$ that yield the same value modulo $N$ in the congruence. Setting $\bsh=\bsk'-\bsk''$ provides the desired solution. Note that, since $N$ is prime, if one component of $\bsh$ were $0$, the other would also necessarily be a multiple of $N$, which is impossible given the bound $\sqrt{N}$ unless the vector is zero. This contradicts $\bsk'\ne \bsk''$.

Now, we construct a fooling function $f\in H_{2,\alpha,\bsone}$ with $\|f\|_{2,\alpha,\bsone}=1$ defined by
\[ f(x_1,x_2) = \frac{e^{2\pi ih_1x_1}-e^{-2\pi ih_2x_2}}{\sqrt{r_{\alpha,1}^2(h_1)+r_{\alpha,1}^2(h_2)}}. \]
For any point $\bsy_n\in P_{N,\bsg}$, the dual lattice property $h_1 g_1\equiv -h_2g_2 \pmod N$ implies 
\begin{align*}
    f(\bsy_n) & = \frac{1}{\sqrt{r_{\alpha,1}^2(h_1)+r_{\alpha,1}^2(h_2)}}\left( e^{2\pi i h_1g_1 n /N}-e^{-2\pi i h_2g_2 n /N}\right) \\
    & = \frac{1}{\sqrt{r_{\alpha,1}^2(h_1)+r_{\alpha,1}^2(h_2)}}\left( e^{2\pi i h_1g_1 n /N}-e^{2\pi i h_1g_1 n /N}\right) = 0.
\end{align*} 
Thus, all discrete Fourier coefficients $\widehat{f}_{N}(\bsk)$ computed via \eqref{eq:aliasing_formula} vanish for any $\bsk\in \Acal_d$. The $L_2$ approximation error is then given by the $L_2$-norm of the function itself:
\begin{align*}
    \|f-A_N^{\text{sing}}(f)\|_{L_2} & = \left(|\widehat{f}(h_1,0)|^2+|\widehat{f}(0,-h_2)|^2\right)^{1/2}\\
    & = \left(\frac{2}{r_{\alpha,1}^2(h_1)+r_{\alpha,1}^2(h_2)}\right)^{1/2}\\
    & \ge \frac{1}{r_{\alpha,1}(\lfloor \sqrt{N}\rfloor)}\ge N^{-\alpha/2}.
\end{align*}
This confirms that any single rank-1 lattice-based algorithm fails to exceed the rate $N^{-\alpha/2}$. As studied in \cite{BKUV17}, this argument applies to a more general class of approximation schemes based on function evaluations along a single rank-1 lattice.

\subsection{A lower bound by Hua and Wang}
To briefly introduce the concept, any real number $x$ can be expanded into a continued fraction of the form $x=[a_0; a_1,a_2,\ldots]$, where $a_0$ is an integer and $a_1,a_2,\ldots$ are positive integers. Truncating this expansion at $a_t$ yields a rational approximation $p_t/q_t\coloneqq [a_0;a_1,\ldots,a_t]$, which is called the $t$-th convergent of $x$.
It is well-known that
\[ \left| x-\frac{p_t}{q_t}\right|\le \frac{1}{q_tq_{t+1}}.\]
Consider a rank-1 lattice $P_{N,\bsg}$ with generating vector $\bsg=(g_1,g_2)\in \{1,\ldots,N-1\}^2$. Let $u=g_1^{-1}g_2 \pmod N$, where $g_1^{-1}$ denotes the multiplicative inverse of $g_1$ modulo $N$, and define $x=u/N$. Let $p_t/q_t$ be the $t$-th convergent of $x$. We define an integer $K$ by
\[ K \coloneqq uq_t-Np_t.\]
From the approximation property of continued fractions, we have
\[ |K|=Nq_t\left| \frac{u}{N}-\frac{p_t}{q_t}\right|\le \frac{N}{q_{t+1}}.\]
Furthermore, from the definition of $K$, we have the congruence $K \equiv u q_t \equiv g_1^{-1}g_2 q_t \pmod N$. Multiplying by $g_1$, we obtain
\[ g_1K\equiv g_2q_t \pmod N, \]
implying that $(K,-q_t)\in P^{\perp}_{N,\bsg}$.

Now, we construct a fooling function $f\in H_{2,\alpha,\bsone}$ with $\|f\|_{2,\alpha,\bsone}=1$ defined by
\begin{align*}
    f(x_1,x_2) = \frac{e^{2\pi iK x_1}-e^{2\pi iq_tx_2}}{\sqrt{r_{\alpha,1}^2(K)+r_{\alpha,1}^2(q_t)}}.
\end{align*}
For any point $\bsy_n\in P_{N,\bsg}$, the above congruence implies that the two exponential terms cancel out, so $f(\bsy_n)=0$. Thus, all discrete Fourier coefficients $\widehat{f}_{N}(\bsk)$ computed via \eqref{eq:aliasing_formula} vanish for any $\bsk\in \Acal_d$. The denominators of the convergents form a sequence satisfying $1=q_0<q_1<\cdots<q_m=N$ for some finite $m$ (since $\gcd(u,N)=1$). Thus, there exists a unique index $t$ such that $q_t\le \sqrt{N}<q_{t+1}$. With this choice, we have $|K|\le N/q_{t+1}<\sqrt{N}$. Therefore,
\begin{align*}
    \|f-A_N^{\text{sing}}(f)\|_{L_2} & = \left(|\widehat{f}(K,0)|^2+|\widehat{f}(0,q_t)|^2\right)^{1/2}\\
    & = \left(\frac{2}{r_{\alpha,1}^2(K)+r_{\alpha,1}^2(q_t)}\right)^{1/2}\\
    & \ge \frac{1}{r_{\alpha,1}(\lfloor \sqrt{N}\rfloor)}\ge N^{-\alpha/2}.
\end{align*}

\bibliographystyle{siam}
\bibliography{ref.bib}

\end{document}